\documentclass{aptpub}

\usepackage[T1]{fontenc}
\usepackage{amsmath,amssymb,mathtools, enumitem}
\usepackage{graphicx}
\usepackage{hyperref}

\allowdisplaybreaks

\usepackage{soul}
\usepackage{color}
\usepackage{xcolor}

\usepackage{pgfplots}

\newcommand{\K}{\mathcal{K}}

\authornames{Jesper M{\o}ller and Eliza O'Reilly}
\shorttitle{Couplings and repulsiveness}

\pgfplotsset{compat=1.5}

\begin{document}
	\title{Couplings for determinantal point processes and their reduced Palm distributions with a view to quantifying repulsiveness}

	\authorone[Aalborg University]{Jesper M{\O}ller}
	\authortwo[California Institute of Technology]{Eliza O'Reilly}
	
	\addressone{Department of Mathematical Sciences, Aalborg University, Skjernvej 4A, 9220 Aalborg {\O}st, Denmark.}
	\addresstwo{Computing and Mathematical Sciences,
	   California Institute of Technology,
           1200 E. California Blvd. MC 305-16,
	   Pasadena, CA 91125,
           USA.}
	
\begin{abstract}
 	%\textcolor{red}{For a determinantal point process $X$ with a kernel $K$ whose spectrum is strictly less than one, Andr{\'e} Goldman has established a coupling to its reduced Palm process $X^u$ at a point $u$ with $K(u,u)>0$ so that in distribution $X^u$ is obtained by removing a finite number of points from $X$. The intensity function of the difference $X\setminus X^u$ is known, but apart from special cases the distribution of $X\setminus X^u$ is unknown. Considering the restriction $X_S$ of $X$ to any compact set $S$, we establish a coupling of $X_S$ and its reduced Palm process $X^u_S$ so that the difference is at most one point. Specifically, we assume $K$ restricted to $S\times S$ is either (i) a projection or (ii) has spectrum strictly less than one. In case of (i), we have in distribution that $X^u_S$ is obtained by removing one point from $X_S$, and we can specify the distribution of this point. In case of  (ii), in distribution we obtain $X^u_S$ either by moving one point in $X$ or by removing one point from $X_S$, and to a certain extent we can describe the distribution of these points. We discuss how Goldman's and our results can be used for quantifying repulsiveness in $X$.}
For a determinantal point process $X$ with a kernel $K$ whose spectrum is strictly less than one, Andr{\'e} Goldman has established a coupling to its reduced Palm process $X^u$ at a point $u$ with $K(u,u)>0$ so that almost surely $X^u$ is obtained by removing a finite number of points from $X$. We sharpen this result, assuming weaker conditions and establishing that $X^u$ can be obtained by removing at most one point from $X$, where we specify the distribution of the difference $\xi_u:=X\setminus X^u$. This is used for discussing the degree of repulsiveness in DPPs in terms of $\xi_u$, including Ginibre point processes and other specific parametric models for DPPs.

\end{abstract}


\msc{60G55; 60J20}{60D05; 62M30} %; 68U20}

\section{Introduction}\label{s:introduction}

Determinantal point processes (DPPs) have been of much interest over the last many years in mathematical physics and probability theory (see e.g.\ \cite{BorodinOlshanski,Hough:etal:09,Macchi:75,shirai:takahash:03,Soshnikov:00} and the references therein) and more recently in other areas, including statistics \cite{LMR15,MR2016}, machine learning \cite{KuleszaTaskar}, signal processing \cite{Dengetal}, and neuroscience \cite{Snoeketal}. They are models for regularity/inhibition/repulsiveness, but there is a trade-off between repulsiveness and intensity \cite{LMR12extended,LMR15} (or see Section~\ref{s:4.1.1}). This paper sheds further light on this issue by studying various couplings between a DPP and its reduced Palm distributions. In particular, we relate our results to the definition of most repulsive stationary DPPs on $\mathbb R^d$ as specified in \cite{LMR15,BiscioLavancier}. However, our results will be given for DPPs in a general setting as given below.

%Section~\ref{s:background} provides the background material needed in this paper: 
Section~\ref{s:def} provides
our general setting for a DPP
 $X$ defined on a locally compact Polish space $\Lambda$ and specified by a so-called kernel $K:\Lambda \times\Lambda \to\mathbb{C}$ 
which satisfies certain mild conditions given in Section~\ref{s:assumptions}. 
Also, for any $u \in \Lambda$ with $K(u,u)>0$,  if $X^u$ follows the reduced Palm distribution of $X$ at $u$ -- intuitively, this is the conditional distribution of $X\setminus\{u\}$ given that $u\in X$ -- then $X^u$ is another DPP; Section~\ref{s:reduced Palm} provides further details. Furthermore, Section~\ref{s:goldman} discusses Goldman's  \cite{Goldman} result that if for any compact set $S\subseteq\Lambda$, denoting $K_S$ the restriction of $K$ to $S\times S$, we have that the spectrum of $K_S$ is $<1$, then 
 $X$ stochastically dominates $X^u$  %(in the sense of Section~\ref{s:stochastic domination}) 
and hence by Strassen's theorem there exists a coupling so that almost surely $X^u\subseteq X$. The difference $\kappa_u:=X\setminus X^u$ is a finite point process with a known intensity function. A straightforward calculation also gives that the mean number of points in $\kappa_u$ is at most 1, see equations \eqref{e:Ku} and \eqref{e:Kintbnd}.
In particular, for a standard Ginibre point process
 \cite{Ginibre}, which is a special case of a DPP in the complex plane, Goldman showed that $\kappa_u$ consists of a single point which follows $N_\mathbb{C}(u,1)$, the complex Gaussian distribution with mean $u$ and unit variance. Section~\ref{s:Pemantle} then discusses some related coupling results due to Pemantle and Peres \cite{Pemantle11}. One of their results implies that in the specific case where $\Lambda$ is finite, there exists a coupling of $X$ and $X^u$ such that $X^u \subseteq X$ and the difference is at most one point. %\eodelete{However, apart from this and other special cases, the distribution of $\kappa_u$ is unknown.} \eoadd{
However, apart from these and other special cases, the distribution of $\kappa_u$ has not been fully characterized.

%Considering the restriction $X_S$ of $X$ to any compact set $S\subseteq \Lambda$, we 
Section~\ref{s:main} shows that these results can be extended: %\eodelete{more can be said: Under weaker conditions than in Goldman's paper,} 
For a DPP $X$ on any locally compact Polish space $\Lambda$, there is a coupling such that almost surely
$X^u\subseteq X$, $\xi_u:=X\setminus X^u$ consists of at most one point, and the distribution of $\xi_u$ can be specified. Note that $\kappa_u$ and $\xi_u$ share the same intensity function, but Goldman did not establish that $\kappa_u$ consists of at most one point.
As in \cite{Goldman} we only verify the existence of our coupling result. We leave it as an open research problem to provide a specific coupling construction or simulation procedure for $(X,X^u)$ (restricted to a compact subset of $\Lambda$), hence extending the simulation algorithm for a DPP \cite{Hough:etal:06,LMR12extended,LMR15,MR2016}.
%We leave it as an open research problem to provide a specific coupling construction or simulation procedure for $(X,X^u)$ (restricted to a compact subset of $\Lambda$); possibly this may provide a faster simulation algorithm than in \cite{Hough:etal:06,LMR12extended,LMR15,MR2016}.

Section~\ref{s:discussion} discusses how our coupling result can be used for describing the repulsiveness in a DPP, including when considering the notion of a globally most repulsive DPP by which we mean that for all $u\in\Lambda$ with $K(u,u)>0$, almost surely
$\xi_u$ has one point. 
For example, if $X$ is a standard Ginibre point process, we obtain a similar result as in \cite{Goldman}: $X$ is a globally most repulsive DPP and the point in $\xi_u$ follows $N_\mathbb{C}(u,1)$. 
In particular, we show that our definitions extend those given in \cite{LMR15,BiscioLavancier}   
for stationary DPPs on $\mathbb{R}^d$, and by considering the distribution of the point in $\xi_u$, we demonstrate how the range of repulsion can differ with DPPs that have the same intensity and the same global repulsiveness. Moreover, we consider the cases of a finite set $\Lambda$ and when we have a stationary DPP defined on $\Lambda=\mathbb R^d$. Finally, we compare with globally most repulsive isotropic DPPs on $\mathbb{S}^d$, the $d$-dimensional unit sphere in $\mathbb{R}^{d+1}$, as studied in \cite{Moelleretal18}.

\section{Background}\label{s:background}

This section provides the background material needed in this paper.

\subsection{Setting}\label{s:setting} 

Below we give the definition of a DPP, specify our assumptions, and recall that the reduced Palm distribution of a DPP is another DPP.

\subsubsection{Definition of a DPP}\label{s:def}

Let $X$ be a point process defined on a locally compact Polish space $\Lambda$ equipped with its Borel $\sigma$-algebra $\mathcal B$ and a Radon measure $\nu$ which is used as a reference measure in the following.
%\footnote{Do we need a metric? So that "Polish space" means "Polish metric space" or equivalently "complete separable metric space".} 
 We assume that $X$ is a DPP with kernel $K$ which by definition means the following. 
First, $X$ has no multiple points, so dependent on the context we view $X$ as a random subset of $\Lambda$ or as a random counting measure, and we let $X(B)$ denote the cardinality of $X_B:=X\cap B$ for $B\in\mathcal B$. Second, $K$ is a complex function defined on  
 $K:\Lambda^2\mapsto\mathbb C$. Third, 
for any $n=1,2,\ldots$ and any mutually disjoint bounded sets $B_1,\ldots,B_n\in\mathcal B$,  
\[\mathrm E\left[X\left(B_1\right)\cdots X\left(B_n\right)\right]=\int_{B_1\times\cdots\times B_n}
\det\left\{K\left(u_i,u_j\right)\right\}_{i,j=1}^n
\,\mathrm d\nu^n\left(u_1,\ldots,u_n\right)\]
is finite, 
where $\nu^n$ denotes the $n$-fold product measure of $\nu$. 
This means that 
$X$ has $n$-th order intensity function $\rho(u_1,\ldots,u_n)$ (also sometimes in the literature called $n$-th order correlation function) given by the determinant
\begin{equation}\label{e:defDPP}
\rho\left(u_1,\ldots,u_n\right)=\det\left\{K\left(u_i,u_j\right)\right\}_{i,j=1}^n,\qquad  u_1,\ldots,u_n\in\Lambda,
\end{equation}
and this function is locally integrable. 
In particular, $\rho(u)=K(u,u)$ is the intensity function of $X$,
and when $B\in\mathcal B$ is bounded almost surely $X_B$ is finite.

In the special case where $K(u,v)=0$ whenever $u\not=v$, the DPP $X$ is just a Poisson process with intensity function $\rho(u)$ conditioned on that there are no multiple points in $X$ (if $\nu$ is diffuse, it is implicit that there are no multiple points). For other examples 
when $\Lambda$ is a countable set and $\nu$ is the counting measure, see \cite{KuleszaTaskar}; when $\Lambda=\mathbb R^d$ and $\nu$ is the Lebesgue measure, see \cite{Hough:etal:09,LMR15}; 
and when 
$\Lambda=\mathbb S^d$ (the $d$-dimensional unit sphere) and $\nu$ is the surface/Lebesgue measure, see \cite{Moelleretal18}. Examples are also given in Section~\ref{s:examplesDPP}. 

%Russ' comment about the Poisson process caused that I moved the first lines in Section 3 to be the following paragraph.
From \eqref{e:defDPP} and 
the fact that the determinant of a complex covariance matrix is less than or equal to the product of its diagonal elements
we obtain that 
\[\rho\left(u_1,\ldots,u_n\right)\le\prod_{i=1}^n\rho\left(u_i\right),\]
where the equality holds if and only if $X$ is a Poisson process.
%\textcolor{blue}{
Thus, apart from the case of a Poisson process, the counts $X(A)$ and $X(B)$ are negatively correlated whenever $A,B\in\mathcal{B}$ are disjoint.
%\textcolor{red}{This inequality shows that DPPs are indeed repulsive. }

\subsubsection{Assumptions}\label{s:assumptions}

We always make the following assumptions (a)--(c):  
\begin{enumerate}
\item[(a)] $K$ is Hermitian, that is, $K(u,v)=\overline{K(v,u)}$ for all $u,v\in\Lambda$;
\item[(b)] $K$ is locally square integrable, that is,  for any compact set $S\subseteq\Lambda$, the double integral $\int_S \int_S |K(u,v)|^2\,\mathrm d\nu(u)\,\mathrm d\nu(v)$ is finite; 
%Russ' comments did not indicate that the change in (b) (the addition of the word "locally") should be done but in fact it is all we need.
\item[(c)] $K$ is of locally trace class, that is, for any compact set $S\subseteq\Lambda$, the integral $\int_S K(u,u)\,\mathrm d\nu(u)$ is finite.
\end{enumerate}
%where indeed (c) was already assumed in Section~\ref{s:def}.
%Russ' comment about the Poisson process caused a reformulation of the following two sentences.
By Mercer's theorem, excluding a $\nu^2$-nullset, 
this ensures the existence of a spectral representation for the kernel restricted to any compact set $S\subseteq\Lambda$:
%Then, % $\nu(S) < \infty$,  and   
Ignoring a $\nu^2$-nullset, we can redefine $K$ on $S\times S$ by 
\begin{align}\label{e:Mercer}
K(u,v) = \sum_{k=1}^{\infty} \lambda_k^S \phi_k^S(u) \overline{\phi_k^S(v)}\qquad  u,v\in S,
\end{align}
where the eigenvalues $\lambda_k^S$ are real numbers and the eigenfunctions $\phi_k^S$ constitute an orthonormal basis of $L^2(S)$,   cf. Section~4.2.1 in \cite{Hough:etal:09}. 
%Russ' comments did not indicate a need for the addition of the following sentence, so this is something I have added.
Here, for any $B\in\mathcal B$, $L^2(B)=L^2(B,\nu)$ is the space of square integrable complex functions w.r.t.\ $\nu$ restricted to $B$. 
%Here, we suppress in the notation the dependence of $S$. , $0$ is the only possible accumulation point of the eigenvalues, and $\sum_{k=1}^\infty\lambda_k^S<\infty$ 
%Russ' comments did not cause me to move "We denote $K$ restricted to $S\times S$ by $K_S$" to the beginning of the second paragraph below; I just did it because it is more natural.
Note that (c)  means $\mathrm E X(S)=\sum_{k=1}^{\infty} \lambda_k^S<\infty$. 
Thus,  when $B\in\mathcal B$ is bounded, almost surely $X_B$ is finite.
%Russ' comment about the Poisson process caused the addition of the following sentence.
When $\nu$ is diffuse, as we are redefining $K$ by \eqref{e:Mercer} we have effectively excluded the special case of the Poisson process (i.e.\ when $K$ is 0 off the diagonal) because all the eigenvalues in \eqref{e:Mercer} are then 0; however, as shown later, it will still make sense to consider the Poisson process when quantifying repulsiveness in DPPs. 

We also always assume that  
\begin{enumerate}
\item[(d)] for any compact set $S\subseteq\Lambda$, all eigenvalues satisfy $0\le\lambda_k^S\le 1$.
\end{enumerate}
% In comparison, Goldman \cite{Goldman} assumed that each eigenvalue is strictly less than one. 
In fact, under (a)--(c), the existence of the DPP with kernel $K$ is equivalent to (d) (see e.g.\ Theorem~4.5.5 in \cite{Hough:etal:09}), and the DPP is then unique (Lemma~4.2.6 in \cite{Hough:etal:09}). If $\Lambda=\mathbb R^d$, $\nu$ is the Lebesgue measure, and $K(u,v)=K_0(u-v)$ is stationary, where $K_0\in L^2(\mathbb R^d)$ and $K_0$ is continuous, 
we denote the Fourier transform of $K_0$ by $\hat K_0$. Then 
(d) is equivalent to $0\le\hat K_0\le1$ (Proposition~3.1 in \cite{Hough:etal:09}). 

%Russ' comments did not cause the addition of text at the beginning of the following sentence, cf.\ my previous comment.
Recalling that $K_S$ is the restriction of $K$ to $S\times S$, we sometimes consider one of the following conditions:
\begin{enumerate}
\item[(e)] For a given compact set $S\subseteq\Lambda$, $K_S$ is a projection of finite rank $n$. 
%Thus, without loss of generality, we can assume
%\begin{equation*}%%\label{e:spectral projection}
%K(v,w)=\sum_{k=1}^n\phi^S_k(v)\overline{\phi^S_k(w)},\qquad v,w\in S.
%\end{equation*}
\item[(f)] 
 For all compact $S\subseteq\Lambda$, all eigenvalues satisfy that $\lambda_k^S<1$.
\end{enumerate}

\subsubsection{Reduced Palm distributions}\label{s:reduced Palm}

Consider an arbitrary point $u\in \Lambda$ with $\rho(u)>0$. Recall that the reduced Palm distribution of $X$ at $u$ is a point process $X^u$ on $\Lambda$ with $n$-th order intensity function
\[\rho^u(u_1,\ldots,u_n)=\rho(u,u_1,\ldots,u_n)/\rho(u).\]
This combined with \eqref{e:defDPP} 
easily shows that $X^u$ is a DPP with kernel
\begin{equation}\label{e:Ku}
K^{u}(v,w)=K(v,w)-\frac{K(v,u)K(u,w)}{K(u,u)}\qquad v,w\in\Lambda,
\end{equation}
see Theorem~6.5 in \cite{shirai:takahash:03}.
% (more precisely, the reduced Palm distributions are uniquely given by such kernels for $\nu$ almost all $u\in \Omega$ with $K(u,u)>0$). 
%Also recall Schur's determinant identity:
%\begin{align}\label{e:det_identity}
%\det \left(\begin{array}{cc} A & B \\ C & D\end{array} \right)=\det(D)\det\left(A-B D^{-1}C\right)
%\end{align}
%provided $D$ is invertible.
For any compact set $S\subseteq\Lambda$, it follows that the restriction $X^u_S:=X^u\cap S$ follows the reduced Palm distribution of $X_S$ at $u$.

\subsection{Goldman's results}\label{s:goldman}

Goldman \cite{Goldman} made similar assumptions as in our assumptions (a)-(d), and in addition he assumed condition (f) throughout his paper. Two of his main results were the following.

\begin{enumerate}
\item[(G1)]
For any $u\in\Lambda$ with $K(u,u)>0$, there is a coupling of $X$ and $X^u$ so that almost surely $X^u\subseteq X$. 
\item[(G2)] Suppose $X$ is a standard Ginibre point process, that is, the DPP on $\Lambda=\mathbb C\equiv\mathbb R^2$, with $\nu$ being Lebesgue measure, and with kernel 
\begin{equation}\label{e:Ginibre-standard-kernel}
%K_{\mathrm{Gin}}
K(v,w)=\frac{1}{\pi}\exp\left(v\overline{w}-\frac{|v|^2+|w|^2}{2}\right),\qquad v,w\in\mathbb C.
\end{equation}
%Russ' comment about a typo caused me to correct $X^u\setminus X$ to $X\setminus X^u$ in the following sentence.
Then, for the coupling in (G1) and any $u\in\mathbb C$, $X\setminus X^u$ consists of a single point which follows $N_\mathbb{C}(u,1)$.
\end{enumerate}

It follows from (G1) and \eqref{e:Ku} that $\kappa_u:=X\setminus X^u$ is a finite point process with intensity function
\begin{equation}\label{e:rhokappau}
\rho_{\kappa_u}(v)={|K(u,v)|^2}/{K(u,u)},\qquad v\in\Lambda.
\end{equation} 
Note that the standard  Ginibre point process is stationary and isotropic with intensity $1/\pi$,  
 but its kernel is not of the form $K(u,v) = K_0(\|u-v\|)$. In accordance with (G2), combining \eqref{e:Ginibre-standard-kernel} and \eqref{e:rhokappau}, $\rho_{\kappa_u}$ is immediately seen to be the density of $N_\mathbb{C}(u,1)$.

\subsection{Pemantle and Peres' results}\label{s:Pemantle}

%\eoadd{Pemantle and Peres \cite{Pemantle11} studied probability measures on $\{0,1\}^n$ satisfying a negative dependence property called the strong Rayleigh property. This class of probability measures was introduced in \cite{Liggett09}, where it was also shown that determinantal point processes on a finite set satisfy this property. In \cite{Pemantle11}, the authors define a notion called stochastic covering and the stochastic covering property, which can be defined as follows. Letting $X$ and $Y$ be point processes, $X$ is said to \textit{cover} $Y$ if there is a coupling $(X,Y)$ such that $X = Y$ or their difference is one point. Now, consider a point process $X$ on a finite set $\Lambda = \{1, \ldots, n\}$ and consider the probability measure $\nu$ on $\{0,1\}^n$ induced by $X$. Then, $X$ is said to have the the stochastic covering property if for all pairs of subsets $S \subseteq \{1, \ldots, n\}$ and $x, y \in \{0,1\}^S$ such that $x > y$ coordinate-wise, $(\nu | Y_i = x_i, i \in S)$ is covered by $(\nu | Y_i = y_i, i \in S)$. In this setting, the stochastic covering property implies that $X$ stochastically covers $X^u$.}
%on $\Lambda$ with $k$ points almost surely for some finite $k >0$ 
%$S_1, S_2 \subseteq \{1, \ldots, n\}$ such that $S_2 \subset S_1$, $(X | S_2 \subseteq X)$ is covered by $(X | S_1 \subseteq X)$.}
% $(j = 0, \ldots k-1$, $X^{u_1, \dots, u_{j}}$ covers $X^{u_1, \ldots, u_{j+1}}$.}

%\eoadd{
Pemantle and Peres \cite{Pemantle11} studied probability measures on $\{0,1\}^n$ satisfying a negative dependence property called the strong Rayleigh property. This class of probability measures was introduced in \cite{Liggett09}, where it was also shown that determinantal point processes on a finite set satisfy this property. In \cite{Pemantle11}, the authors define notions called stochastic covering and the stochastic covering property, which can be defined as follows. %Let $x, y \subset \{1, \ldots, n\}$.
Letting $X$ and $Y$ be simple point processes, $X$ is said to \textit{stochastically cover} $Y$ if there is a coupling $(X,Y)$ such that $X = Y$ or their difference $X\setminus Y$ is one point. Now, consider a simple point process $X$ on a finite set $\Lambda = \{1, \ldots, n\}$. Then, $X$ is said to have the stochastic covering property if the following holds. If $u\in x\subseteq S\subseteq\Lambda$ and we set $y=x\setminus \{u\}$, then the point process $X_{S^c}$ conditioned on $X_S=x$ is stochastically covered by the point process $X_{S^c}$ conditioned on $X_S=y$. This property implies for $u \in \{1, \ldots, n\}$ (letting $S =x= \{u\}$) that the point process $X\setminus \{u\}$ conditioned on $u\notin X$ stochastically covers $X^u$, and in turn that $X$ stochastically covers $X^u$.

Proposition 2.2 in \cite{Pemantle11} states that for a probability measure on $\{0,1\}^n$, the strong Rayleigh property implies the stochastic covering property, and thus determinantal point processes on $\Lambda = \{1, \ldots, n\}$ satisfy this property. The authors discuss extensions to the case where $\Lambda$ is continuous, and in particular they extend their Proposition 2.3 to this case. However, a generalization of their Proposition 2.2 to the case of continuous determinantal point processes does not appear in the most recent version \cite{Pemantle11}. As kindly pointed out by a referee, in the first version of this paper on arXiv \cite{Pemantle11arxiv}, the authors claim $X$ stochastically covers $X^u$ in the continuous case as well, and the main idea of our proof of Theorem \ref{t:main-result} below is outlined. However, their justification is not complete for our general setting.

%Proposition 2.2 in \cite{Pemantle11} states that for a probability measure on $\{0,1\}^n$, the strong Rayleigh property implies the stochastic covering property, and thus determinantal point processes on $\Lambda = \{1, \ldots, n\}$ satisfy this property. 
%The authors discuss extensions to the case where $\Lambda$ is continuous, and in particular they extend their Proposition 2.3 to this case. However, a generalization of their Proposition 2.2 to the case of continuous determinantal point processes does not appear in the published version. \eoadd{In the first version of this paper on arXiv, the authors claim $X$ stochastically covers $X^u$ in the continuous case as well, and the main idea of our proof is outlined. However, the justification is not complete for our general setting.} % We present here a complete proof of this claim for a general locally compact Polish space $\Lambda$.}}

\section{Main result}\label{s:main}

The theorem below is our main result which is sharpening Goldman's result (G1) in %\eodelete{two ways: We do not assume condition (f) and we establish a coupling so that $X$ contains $X^u$, the difference is at most one point, and we can completely describe the distribution of this difference.} 
that we do not assume condition (f) and we establish a coupling so that $X$ stochastically covers $X^u$. It also sharpens Pemantle and Peres' result since it holds for a %by showing that $X$ stochastically covers $X^u$ for a 
general locally compact Polish space $\Lambda$. In addition, we completely describe the distribution of the difference $X \backslash X^u$. In the proof of the theorem we use basic results and definitions for operators on the Hilbert space $\mathcal{L}^2(\Lambda)$, see e.g.\ \cite{SzNagy-etal,Paulsen}. An outline of the proof is as follows. First, we dilate the operator associated to the DPP $X$ to a projection operator on the union of two copies of $\Lambda$. Second, we use the existence of a coupling for projection operators in Lemma~\ref{l1}. Finally, we compress back down to $\Lambda$ to obtain the desired coupling. 

We use the following special result established under condition (e) and where $\nu_S$ denotes the restriction of the reference measure $\nu$ to a compact set $S\subseteq\Lambda$.

\begin{lemma}\label{l1}
Assume $S\subseteq\Lambda$ is compact and 
let $\{\phi^S_k\}_{k=1}^n$ be an orthonormal set of functions in $L^2(S)$ with $1\le n<\infty$. Let 
$X$ and $Y$ be DPPs with kernels $K$ and $L$, respectively, so that 
\[ K(v,w) = \sum_{k=1}^n \phi^S_k(v) \overline{\phi^S_k(w)}, \qquad
L(v,w) = \sum_{k=1}^{n-1} \phi^S_k(v) \overline{\phi^S_k(w)}, \qquad v,w \in S\] 
(setting $L(v,w)=0$ if $n=1$). 
Then there exists a monotone coupling of $Y_S$ w.r.t.\ $X_S$ such that almost surely $Y_S\subset X_S$, $\eta_S:=X_S \setminus Y_S$ consists of one point, and the point in $\eta_S$ has density $|\phi^S_n(\cdot)|^2$ w.r.t.\ $\nu_S$.
\end{lemma}

\begin{proof}
Observe that $K$ and $L$ are the kernels of finite dimensional projections, a special case of trace-class positive contractions, and the difference, 
\[K(v,w) - L(v,w) = \phi^S_n(v)\overline{\phi^S_n(w)}, \qquad v,w \in S,\] 
is a positive definite kernel.
Thus, by Theorem~3.8 in \cite{Lyons:14}, $X_S$ stochastically dominates $Y_S$. Therefore, there is a coupling such that almost surely $Y_S\subseteq X_S$. As $Y_S$ has cardinality one less than $X_S$, almost surely $\eta_S := X_S\setminus Y_S$ consists of one point. Finally, for any Borel set $A \subseteq S$,
\begin{equation*}
\mathrm{P}( \eta_S \cap A\not=\emptyset) = \mathrm{E}\left[1_{\{X(A) - Y(A) = 1\}}\right] = \mathrm{E}[X(A)] - \mathrm{E}[Y(A)] = \int_A |\phi^S_n(v)|^2\,\mathrm d\nu(v).
\end{equation*}
%Thus, the result follows.
\end{proof}

%Let $\K$ be the operator on $L^2(\Lambda, \nu)$ with kernel 
%\[K(u,v) = \sum_{k \geq 0} \lambda_k \phi^\Lambda_k(u)\overline{\phi^\Lambda_k(v)}.\] 
%Then, the operator $\mathcal{K}(\mathcal{I} - \mathcal{K})$ has a kernel with spectral representation
%\[ \sum_{k \geq 0} \lambda_k(1 - \lambda_k) \phi^\Lambda_k(u)\overline{\phi^\Lambda_k(v)},\]
%and the square room of the operator, $\sqrt{\mathcal{K}(\mathcal{I} - \mathcal{K})}$, has spectral representation
%\[ \sum_{k \geq 0} \sqrt{\lambda_k(1 - \lambda_k)} \phi^\Lambda_k(u)\overline{\phi^\Lambda_k(v)}.\]

Denote $\|\cdot\|_2$ the usual norm on $\mathcal{L}^2(\Lambda)$ w.r.t.\ $\nu$.

\begin{theorem}\label{t:main-result} Let $X$ be a DPP on $\Lambda$ with kernel $K$ satisfying conditions (a)--(d).
For any $u \in \Lambda$ with $K(u,u)>0$, there exists a coupling of $X$ and $X^u$ such that almost surely $X^u\subseteq X$ and $\xi_u:=X\setminus X^u$ consists of at most one point. 
%\[X = X^u \cup \xi^u, \qquad X^u \cap \xi^u = \emptyset,\]
%where $\xi^u$ is a point process in $\Lambda$ with at most one point. 
We have
\begin{equation}\label{e:def-pu}
p_u:=\mathrm{P}(\xi_u \neq \emptyset ) = \frac{1}{K(u,u)}\int |K(u,v)|^2\,\mathrm d\nu(v),
\end{equation}
and conditioned on $\xi_u \neq \emptyset$ the point in $\xi_u$ has density 
\begin{equation}\label{e:def-fu}
f_u(v):={|K(u,v)|^2}/{\|K(u, \cdot)\|_2^2},\qquad v\in\Lambda,
\end{equation} 
w.r.t.\ $\nu$.
%for any $B \in \mathcal{B}$,
%\[ \mathrm{P}(\xi^u \cap B\not=\emptyset\, |\, \xi^u \neq \emptyset ) = \int_B \frac{ |K(u,v)|^2}{\|K(u, \cdot)\|_2^2}\,\mathrm d\nu(v).\]
\end{theorem}

Compared to Goldman's result (G1), we also have $p_u=\mathrm P(\kappa_u\not=\emptyset)$ and $f_u$ is the conditional density of a point in $\kappa_u$ given that $\kappa_u\not=\emptyset$, cf.\ \eqref{e:rhokappau}--\eqref{e:def-fu}. 

\begin{proof} 
We begin by describing the procedure given in Lyons' paper \cite[Section 3.3]{Lyons:14} for dilating a locally trace class operator to a locally trace class orthogonal projection. Denote $\K$ the locally trace class operator on $\mathcal{L}^2(\Lambda)$ with kernel $K$. 
%\eodelete{As in section 3.3 in \cite{Lyons:14},} 
Consider the dilation of $\K$ given by
\[\mathcal{Q} := \begin{bmatrix} \K & \mathcal{L} \\  \mathcal{L} & \mathcal{I} - \K \end{bmatrix},\]
where $\mathcal{L} := \sqrt{\K (\mathcal{I} - \K)}$. Then, since $Q = Q^2$, $Q$ is an orthogonal projection on $L^2(\Lambda, \nu) \oplus L^2(\Lambda_0, \nu)$, where $\Lambda_0$ is a disjoint identical copy of $\Lambda$. If $\Lambda$ is discrete, then $Q$ is clearly locally trace class, since any compact set of a discrete space is finite. %, and thus any operator on a discrete base space is locally trace class. 
If $\Lambda$ is not discrete, consider the operator 
\[\mathcal{Q}' := \begin{bmatrix}\mathcal{I} & 0 \\ 0 & \mathcal{U} \end{bmatrix}^*\mathcal{Q}\begin{bmatrix}\mathcal{I} & 0 \\ 0 & \mathcal{U} \end{bmatrix} = \begin{bmatrix} \K & \mathcal{L}\mathcal{U} \\  \mathcal{U}^*\mathcal{L} & \mathcal{U}^*(\mathcal{I} - \K)\mathcal{U} \end{bmatrix},\] 
where $\mathcal{U}$ is a unitary operator from $\ell^2(\Lambda_0')$ to $L^2(\Lambda_0, \nu)$ for some countably infinite space $\Lambda_0'$. The operator $\mathcal{U}$ exists since any two infinite dimensional separable Hilbert spaces are unitarily equivalent. The operator $\mathcal{Q}'$ is an orthogonal projection on $L^2(\Lambda, \nu) \oplus \ell^2(\Lambda_0')$, and $\K$ is the compression of $\mathcal{Q}'$ to $\Lambda$. Further, $\mathcal{Q}'$ is also locally trace class, because $\mathcal{K}$ is locally trace class on $L^2(\Lambda, \nu)$ by assumption, and %\textcolor{red}{$\mathcal{U}^*\mathcal{Q}\mathcal{U}$ is} \textcolor{blue}{
all operators on $\ell^2(\Lambda_0')$ are locally of trace class since $\Lambda_0'$ is discrete. Thus, $\mathcal{Q}'$ defines a projection DPP $Y_Q$ on the union $\Lambda \cup \Lambda_0'$. 

First, assume that $\Lambda$ is compact. Then, the kernel of the operator $\mathcal{K}$ satisfies
\[K(v,w) = \sum_{k \geq 1} \lambda^\Lambda_k \phi^\Lambda_k(v)\overline{\phi^\Lambda_k(w)}, \qquad v,w \in \Lambda,\]
where $\{\phi^\Lambda_k\}$ is an orthonormal basis for $L^2(\Lambda)$, $\lambda^\Lambda_k \in [0,1]$ for all $k$, and $\sum_{k \geq 1} \lambda^\Lambda_k < \infty$. 
Also, the kernel for the operator $\mathcal{L}$ is then given by
\[L(v, w) = \sum_{k \geq 1} \sqrt{\lambda^\Lambda_k(1-\lambda^\Lambda_k)} \phi^\Lambda_k(v)\overline{\phi^\Lambda_k(w)}.\]

Note that
\begin{align*}
\mathcal{L}(L(\cdot, u))(w) = \int_{\Lambda} L(w, v)L(v, u)\,\mathrm d\nu(v) =\sum_{k \geq 1} \lambda^\Lambda_k(1-\lambda^\Lambda_k) \phi^\Lambda_k(w)\overline{\phi^\Lambda_k(u)},
\end{align*}
and 
\begin{align*}
\mathcal{K}(K(\cdot, u))(w) = \int_{\Lambda} K(w, v)K(v, u)\,\mathrm d\nu(v) =\sum_{k \geq 1} \left(\lambda^\Lambda_k\right)^2 \phi^\Lambda_k(w)\overline{\phi^\Lambda_k(u)}.
\end{align*}
Hence, $\mathcal{K}(K(\cdot, u)) + \mathcal{L}(L(\cdot, u)) = K(\cdot, u)$. Also,
\begin{align*}
\mathcal{L}(K(\cdot, u))(w) = \int_{\Lambda} L(w, v)K(v, u)\,\mathrm d\nu(v) =\sum_{k \geq 1} \lambda^\Lambda_k\sqrt{\lambda^\Lambda_k(1-\lambda^\Lambda_k)} \phi^\Lambda_k(w)\overline{\phi^\Lambda_k(u)}
\end{align*}
and 
\begin{align*}
\mathcal{K}(L(\cdot, u))(w) = \int_{\Lambda} K(w, v)L(v, u)\,\mathrm d\nu(v) =\sum_{k \geq 1}  \lambda^\Lambda_k\sqrt{\lambda^\Lambda_k(1-\lambda^\Lambda_k)}  \phi^\Lambda_k(w)\overline{\phi^\Lambda_k(u)},
\end{align*}
and so $\mathcal{L}(K(\cdot, u)) = \mathcal{K}(L(\cdot, u))$.
Consequently, for fixed $u \in \Lambda$,
\[\psi_u(\cdot) := \begin{bmatrix} \frac{K(\cdot,u)}{\sqrt{K(u,u)}} \\ \mathcal{U}^*\left(\frac{L(\cdot,u)}{\sqrt{K(u,u)}}\right) \end{bmatrix}\] 
is an eigenvector of the operator $\mathcal{Q}'$. Indeed, since $\mathcal{U}\mathcal{U}^* = \mathcal{I}$ by that fact that $\mathcal{U}$ is unitary, 
\begin{align*}
\mathcal{Q}'(\psi_u(\cdot)) &= \begin{bmatrix}\mathcal{I} & 0 \\ 0 & \mathcal{U} \end{bmatrix}^*Q \begin{bmatrix} \frac{K(\cdot,u)}{\sqrt{K(u,u)}} \\ (\mathcal{U}\mathcal{U}^* )\left(\frac{L(\cdot,u)}{\sqrt{K(u,u)}}\right) \end{bmatrix} \\
&= \begin{bmatrix}\mathcal{I} & 0 \\ 0 & \mathcal{U}^* \end{bmatrix}\begin{bmatrix}  \frac{\mathcal{K}(K(\cdot, u))}{\sqrt{K(u,u)}} + \frac{\mathcal{L}(L(\cdot, u))}{\sqrt{K(u,u)}}  \\  \frac{\mathcal{L}(K(\cdot, u))}{\sqrt{K(u,u)}} + \frac{(\mathcal{I} - \K)(L(\cdot, u))}{\sqrt{K(u,u)}} \end{bmatrix} =  \begin{bmatrix} \frac{K(\cdot,u)}{\sqrt{K(u,u)}} \\ \mathcal{U}^*\left(\frac{L(\cdot ,u)}{\sqrt{K(u,u)}} \right)\end{bmatrix} = \psi_u(\cdot). 
\end{align*}
%So $\psi_u$ is in the range of $\mathcal{Q}'$ and 
Then, we can define the projection
\[\mathcal{Q}_u' := \mathcal{Q}' - P_{\psi_u}, \]% \begin{bmatrix} \frac{K(v,u)}{\sqrt{K(u,u)}} \\ \frac{L(w,u)}{\sqrt{K(u,u)}} \end{bmatrix} \begin{bmatrix} \frac{K(v,u)}{\sqrt{K(u,u)}} & \frac{L(w,u)}{\sqrt{K(u,u)}} \end{bmatrix},\]
%where 
%\[P_{\psi_u} = \begin{bmatrix} \frac{K(\cdot,u)}{\sqrt{K(u,u)}} \\ \mathcal{U}^*\left(\frac{L(\cdot,u)}{\sqrt{K(u,u)}}\right) \end{bmatrix} \begin{bmatrix} \frac{K(\cdot,u)}{\sqrt{K(u,u)}} \\ \mathcal{U}^*\left(\frac{L(\cdot,u)}{\sqrt{K(u,u)}}\right) \end{bmatrix}^T\]
where $P_{\psi_u}$ is the projection operator on $L^2(\Lambda, \nu) \oplus \ell^2(\Lambda_0')$ onto the span of $\psi_u$. This projection operator is also locally trace class since it is the difference of locally trace class operators.
Then we can define the projection DPP $Y_Q^u$ on $\Lambda \cup \Lambda_0'$ associated with $\mathcal{Q}_u'$. If $\mathcal{Q}'$ has finite rank, then $\mathcal{Q}'$ and $\mathcal{Q}_u'$ have corresponding kernels
\[Q' = \sum_{k=0}^n q_k q_k^T\quad \text{ and }\quad Q_{u}' = \sum_{k=1}^n q_k q_k^T,\]
where $n < \infty$, $\{q_k\}_{k=1}^n$ is an orthonormal set, and $q_0 :=  \psi_u$. Applying Lemma~\ref{l1} then gives the result. 

Now, assume $\mathcal{Q}'$ projects onto an infinite dimensional subspace of $L^2(\Lambda, \nu) \oplus \ell^2(\Lambda_0')$ and let $\{q_k\}_{k=0}^{\infty}$ be an orthonormal basis for the range of $\mathcal{Q}'$, where $q_0 :=  \psi_u$. For each positive integer $M$, define the finite dimensional projection kernels
\[Q_M' = \sum_{k=0}^M q_k q_k^T\quad \text{ and }\quad Q_{M, u}' = \sum_{k=1}^M q_k q_k^T,\]
and let $Y_{Q_M}$ and $Y^u_{Q_M}$ be the corresponding projection DPPs. By Lemma \ref{l1}, there is a coupling of $Y_{Q_M}$ and $Y_{Q_M}^u$ such that almost surely
$Y_{Q_M} \supset Y^u_{Q_M}$, 
where $\xi_{Q_M}^u:=Y_{Q_M} \setminus Y^u_{Q_M}$ consists of one point which has density $|\psi_u(\cdot)|^2$. By the same argument as in the proof of Lemma~20 in \cite{Goldman}, the sequences $Y_{Q_M}$ and $Y^u_{Q_M}$ are tight and converge in distribution to $Y_Q$ and $Y^u_Q$, respectively, as $M \to \infty$. Also, the sequence $(Y^u_{Q_M}, \xi_{Q_M}^u)_{M}$ is tight, and thus a subsequence converges in distribution to $(Y^u_Q, \xi^u_Q)$, where $\xi^u_Q$ consists of one point with density $|\psi_u(\cdot)|^2$, and $Y^u_Q \cup \xi^u_Q$ is equal in distribution to $Y_Q$. %Thus, taking the limit as $M \to \infty$ gives the existence of a coupling of  $Y_{Q}$ and $Y_{Q}^u$ such that almost surely,
%\[Y_{Q} = Y^u_{Q} \cup \xi_{Q}^u,\]
%where $\xi_{Q}^u$ consists of one point.

The projection operator $P_{\psi_u}$ has kernel $\psi_u \psi_u^T$ and the compression of $P_{\psi_u}$ to $\Lambda$ is the integral operator with kernel
\[\frac{K(v,u)K(u,w)}{K(u,u)}.\]
Then, since the compression of $\mathcal{Q}'$ to $\Lambda$ is the operator $\mathcal{K}$, the compression of $\mathcal{Q}_u'$ to $\Lambda$ is the integral operator $\mathcal{K}^u$ with kernel
\[K^u(v,w) = K(v,w) -\frac{K(v,u)K(u,w)}{K(u,u)}.\]
This gives that %Then, by is the operator $\K^u$ associated to $X^u$, and compressing $Q$ gives $\mathcal{K}$, we see that 
$Y_Q \cap \Lambda$ has the same distribution as $X$ and $Y^u_Q \cap \Lambda$ has the same distribution as $X^u$. Thus, almost surely
\[X = X^u \cup \xi_u,\] 
where $\xi_u := \xi_Q^u \cap \Lambda$ and $X^u$ are disjoint. Therefore, we have a coupling of $X$ and $X^u$, where almost surely $X^u \subseteq X$ and the difference is at most one point. 
The probability of $\xi_u \neq \emptyset$ is the probability that $\xi^u_Q$ is in $\Lambda$, and
the density of $\xi^u_Q$ restricted to $\Lambda$ is
\[f_{\xi^u_Q}(v)1_{\{v \in \Lambda\}} = \frac{|K(v,u)|^2}{K(u,u)}\]
w.r.t.\ $\nu$. 
Hence,
\begin{align*}
\mathrm{P}(\xi^u \neq \emptyset) = \mathrm{P}(\xi^u_Q \in \Lambda) = \int\frac{|K(v,u)|^2}{K(u,u)}\,\mathrm d\nu(v)
%=\frac{\sum_{k\geq 1} \lambda_k^2 |\phi_k(u)|^2}{K(u,u)},
\end{align*}
and the density of $\xi_u$ conditioned on $\xi_u \neq \emptyset$ is $f_u(v)  = {|K(v,u)|^2}/{\|K(\cdot,u)\|_2^2}$ w.r.t.\ $\nu$.

Second, if $\Lambda$ is not assumed to be compact,
consider a sequence of compact sets $S_n \subset \Lambda$ such that $\cup_{n=1}^\infty S_n = \Lambda$ and $S_n \subseteq S_{n+1}$ for $n=1,2,\ldots$. For each $n$, using the result above with $\Lambda$ replaced by $S_n$, there exists a coupling of $(X_{S_n}, X_{S_n}^u)$, where almost surely $X_{S_n} = X_{S_n}^u \cup \xi^u_{S_n}$, $\xi^u_{S_n}=X_{S_n}\setminus X_{S_n}^u$ consists of at most one point, 
\begin{equation}\label{e:helpa}
\mathrm{P}\left(\xi^u_{S_n} \neq \emptyset\right) = \int_{S_n}\frac{|K(v,u)|^2}{K(u,u)}\,\mathrm d\nu(v),
\end{equation}
and conditioned on $\xi^u_{S_n} \neq \emptyset$ the density of the point in $\xi^u_{S_n}$ is 
\begin{equation}\label{e:helpb}
f_{u,S_n}(v)=|K(v,u)|^2/\int_{S_n} |K(w,u)|^2\,\mathrm d\nu(w)
\end{equation}
w.r.t.\ $\nu_{S_n}$.
For consistency, 
let $T_1=S_1$ and generate a realization $(y_{T_1},y^u_{T_1})$ of $(Y_{T_1},Y^u_{T_1}):=(X_{S_1},X^u_{S_1})$, 
and for $n=2,3,\ldots$, let $T_n=S_n\setminus S_{n-1}$ and generate a realization $(y_{T_n},y^u_{T_n})$ of $(Y_{T_n},Y^u_{T_n})$ 
which follows the conditional distribution of $(X_{S_n}\setminus S_{n-1},X^u_{S_n}\setminus S_{n-1})$ 
given that $(X_{S_{n}}\cap S_{n-1},X^u_{S_{n}}\cap S_{n-1})=(\cup_{i=1}^{n-1}y_{T_i},\cup_{i=1}^{n-1}y^u_{T_i})$. 
Then $(X,X^u)$ is distributed as $(Y,Y^u):=(\cup_{n=1}^\infty Y_{T_n},\cup_{n=1}^\infty Y^u_{T_n})$, and almost surely, for $n=2,3,\ldots$, $Y_{T_{n-1}}\setminus Y^u_{T_{n-1}}\not=\emptyset$ implies that 
$Y_{T_{n}}\setminus Y^u_{T_{n}}=Y_{T_{n+1}}\setminus Y^u_{T_{n+1}}=\ldots=\emptyset$, and so $\xi_u:=Y\setminus Y^u$ consists of at most one point.  
%take $Y_{S_{n-1}} := X_{S_{n}} \cap S_{n-1}$ and $X^u_{S_{n-1}} := X^u_{S_{n}} \cap S_{n-1}$ for each $n$. This gives that in distribution $X_{S_{n-1}} = X^u_{S_{n-1}} \cup \xi^u_{S_{n-1}}$, where $\xi^u_{S_{n-1}} = \xi^u_{S_{n}} \cap S_{n-1}$, and so also consists of at most one point. Taking $n$ to infinity gives a limiting coupling of DPPs $X = \cup_n X_{S_n}$ and $X^u = \cup_n X_{S_n}^u$ on $\Lambda$ such that $X = X^u \cup \xi^u$, where $\xi^u := \cup_n \xi^u_{S_n}$. Since $\xi^u_{S_{n-1}} \subseteq \xi^u_{S_n}$ for all $n$, $\xi^u$ consists of at most one point. 
The probability that $\xi_u$ is non-empty is, by \eqref{e:helpa},
\[\mathrm{P}(\xi_u \neq \emptyset) =\lim_{n \to \infty}\int_{S_n}\frac{|K(v,u)|^2}{K(u,u)}\,\mathrm d\nu(v) \]
% \begin{align*}
%\mathrm{P}(\xi_u \neq \emptyset) =\,& \lim_{n \to \infty} \mathrm{P}\left(\xi^u_{S_n}\neq \emptyset\right) = \lim_{n \to \infty}\int_{S_n}\frac{|K(v,u)|^2}{K(u,u)}\,\mathrm d\nu(v) 
%%\\=\,&  \int_{\Lambda}\frac{|K(v,u)|^2}{K(u,u)}\,\mathrm d\nu(v).
%\end{align*}
and hence by monotone convergence we obtain \eqref{e:def-pu}.
%Similarly, we obtain
Finally, \eqref{e:def-fu} is obtained in a similar way using \eqref{e:helpb}. 
%see that the density of the point in $\xi_u$ conditioned on $\xi_u \neq \emptyset$ is 
%%$f_{\xi^u | \xi^u \neq \emptyset}(v)  = 
%$\frac{|K(v,u)|^2}{\|K(\cdot,u)\|_{2}^2}$.

%$X_{S_1} \subseteq X_{S_2} \subseteq \cdots \subseteq X$, for each $n$, $X^u_{S_n} \subseteq X_{S_n}$, thus $\cup_n X^u_{S_n} = X^u \subseteq X$, and $X \backslash X^u$ = \cup_n 

\end{proof}

\section{Describing repulsiveness in DPPs}\label{s:discussion}

%Russ' comment about the Poisson process caused that I moved the first 5 lines to just after equation (1).
In this section we use the probability $p_u$ to quantify how repulsive a DPP can be, and we use the density $f_u$ from Theorem~\ref{t:main-result} to describe the repulsive effect of a fixed point contained in a DPP. As mentioned in Section~\ref{s:measure of repulsiveness}, in the case of 
stationary DPPs on $\mathbb R^d$, $p_u$ turns out to agree with a measure for repulsiveness studied in \cite{LMR15,LMR12extended,BiscioLavancier,BaccelliOReilly}, but we are not aware of references where $f_u$ has been considered when discussing repulsiveness in DPPs. Examples of $p_u$ and $f_u$ for specific models of DPPs are given in Section~\ref{s:examplesDPP}.

Note that $X^u$ is the point process where there is a `ghost point' at $u$ that is affecting the remaining points. %\textcolor{red}{From the coupling it is clear that the effect this ghost point has is} \textcolor{blue}{
Using this coupling of $X^u$ and $X$, it is clear that the repulsive effect of a point at location $u$ is characterized by the difference between $X^u$ and the original DPP $X$, where there is no repulsion coming from the location $u$. 
Further, as $X$ and $X^u$ have intensity functions $\rho(\cdot)$ and $\rho(u,\cdot)/\rho(u)$, respectively, $\xi_u=X\setminus X^u$ has intensity function 
\[\rho_u(v):={|K(v,u)|^2}/{K(u,u)},\qquad v\in\Lambda.\]
%\[\rho_u(v):={|K(v,u)|^2}/{\|K(\cdot,u)\|_{2}^2},\qquad v\in\Lambda.\]
This is the intensity function for the points in $X$ `pushed out' by $u$ under the Palm distribution.
It makes also sense to consider $\rho_u$ as the intensity function of $X\setminus X^u$ when $\nu$ is diffuse and $X$ is a Poisson process because then $X=X^u$ and $\rho_u(v)=0$ for $v\not=u$. 

\subsection{A measure of repulsiveness}\label{s:measure of repulsiveness}

 Setting $0/0=0$, recall that
the pair correlation function of $X$ is defined by $g(v,w)=\rho(v,w)/(\rho(v)\rho(w))$ for $v,w\in\Lambda$, so it satisfies
\[1-g(v,w)=|r(v,w)|^2,\qquad v,w\in\Lambda,\]
where $r(v,w)=K(v,w)/\sqrt{K(v,v)K(w,w)}$ is the correlation function obtained from $K$. Note that %as $X$ and $X^u$ have intensity functions $\rho(\cdot)$ and $\rho(u,\cdot)/\rho(u)$, respectively, 
\begin{equation}\label{e:pcf}
\rho_{u}(v)=\rho(v)(1-g(u,v)),\qquad v\in\Lambda.
\end{equation}
%Russ' comment about the trade-off caused that I moved "which shows a trade-off..." to the next paragraph and modified the following text.
%%which shows  a trade-off between intensity and repulsiveness (see also \cite{LMR15,LMR12extended}). 

%and because of \eqref{e:pcf}, when using $\rho_{u}$ to compare repulsiveness in two DPPs, they should share the same intensity function $\rho$ (see also \cite{LMR15,LMR12extended}), but the distribution of the removed point $\xi_u$ is a characteristic of the DPP that is not dependent on $\rho$.  Taking this into account,
%small/high values of $\rho_{u}$ correspond to small/high degree of repulsiveness. 

\subsubsection{Defining a global measure of repulsiveness}\label{s:4.1.1}
As a global measure of repulsiveness in $X$ when having a point of $X$ at $u$, we suggest the probability of $\xi_u\not=\emptyset$, that is, 
%Russ' comment about the trade-off caused me to add a new expression to the following definition of $p_u$.
\[p_u=\int \rho_{u}(v)\,\mathrm d\nu(v)=\int {|K(u,v)|^2}/{K(u,u)}\,\mathrm d\nu(v).\]
%Russ' comment caused me to change the period above to a comma and to change "This" to "which".
%Russ' comment about the trade-off caused me to add the following sentence.
By \eqref{e:pcf}, there is a trade-off between intensity and repulsiveness: If $p_u$ is fixed, we cannot both increase $\rho$ and decrease $g$. Therefore, when using $p_u$ as a measure to compare repulsiveness in two DPPs, they should share the same intensity function $\rho$. Then small/high values of $p_{u}$ correspond to small/high degree of repulsiveness.  For a stationary DPP $X$ on $\mathbb R^d$, 
%{\color{red} \st{apart from a constant (given by the intensity of $X$),}} $p_u$ }}{\color{red} 
$p_u$ agrees with the measure for repulsiveness in DPPs introduced in 
\cite{LMR15,LMR12extended}; see also \cite{BiscioLavancier,BaccelliOReilly}. Indeed this measure is very specific for DPPs as discussed later in Section~\ref{s:rem}. 
%{\color{red} Finally, note that when the intensity function $\rho$ is constant, conditioned on $\xi_u\not=\emptyset$, the density $f_u(v)=\rho_u(v)/p_u$ of the removed point $\xi_u$ is a characteristic of the DPP that is not dependent on the intensity function $\rho$. Jesper: I think this is not true in general.  Consider (14). IF we fix $p_u$, then $\rho = p_u/(\beta \pi)$ and $f_u(v) \sim N(u,\beta)$. Then, there is a trade-off where the larger the intensity is, the small the variance is of the point in $\xi_u$ that is removed.}

%The following proposition shows that $p_u$ is a probability which relates to the spectral decomposition of $K$ restricted to a compact set increasing towards $\Lambda$.  
%
%\begin{prop}\label{p1}
%Assume (a)--(d) and let $u\in\Lambda$ with $K(u,u)>0$. Then $0\le p_u\le1$. Moreover, $p_u=1$ if and only if for an
% increasing sequence of compact sets $S_1\subseteq S_2\subseteq\ldots$ so that $u\in S_1$ and $\Lambda=\cup_{i=1}^\infty S_i$, for all the eigenvalues in the spectral decomposition of $K_{S_i}$ we have that $\lambda^{S_i}_k$ tends to 0 or 1 when $\phi^{S_i}_k(u)\not=0$, or more precisely 
%\[\lim_{i\rightarrow\infty}\sum_k\lambda^{S_i}_k(1-\lambda^{S_i}_k)|\phi^{S_i}_k(u)|^2=0.\]
%\end{prop}

%Consequently, i

\subsubsection{Definition of globally most repulsive DPPs}
If $p_u=1$ for all $u\in\Lambda$ with $K(u,u)>0$, we say that $X$ is a globally most repulsive DPP. This is the case if $K$ is a projection, that is, for all $v,w\in\Lambda$,
\[K(v,w)=\int K(v,y)K(y,w)\,\mathrm d\nu(y).\] 
For short we then say that $X$ is a projection DPP. 
%Russ' comment about the standard Ginibre point process  caused that the following sentence has been changed.
The standard Ginibre point process given by \eqref{e:Ginibre-standard-kernel} is globally most repulsive, and its kernel is indeed a projection; this follows from a straightforward calculation using that $(v,w)\to\exp(v\overline{w})$ is the reproducing kernel of the Bargmann-Fock space equipped with the standard complex Gaussian measure.
At the other end, if $\nu$ is diffuse and $X$ is a Poisson process with intensity function $\rho$, then $p_u=0$ for all $u\in\Lambda$ with $\rho(u)>0$, and so $X$ is a globally least repulsive DPP. 

 If $\Lambda$ is compact, then it follows from the spectral representation \eqref{e:Mercer} and condition (d) that
\begin{align}\label{e:Kintbnd}
\nonumber \int_S|K(u,v)|^2\,\mathrm d\nu(v)
=&
\sum_k\sum_\ell\lambda^S_k\lambda^S_l\phi^S_k(u)\overline{\phi^S_\ell(u)}\int_S\overline{\phi^S_k(v)}\phi^S_\ell(v)\,\mathrm d\nu(v)\\ 
=&\sum_k\left(\lambda^S_k\right)^2|\phi^S_k(u)|^2\le \sum_k\lambda^S_k|\phi^S_k(u)|^2=K(u,u),
\end{align}
and so
\begin{equation}\label{e:compact p}
p_u=\frac{\sum_k\left(\lambda^\Lambda_k\right)^2|\phi^\Lambda_k(u)|^2}{\sum_k\lambda^\Lambda_k|\phi^\Lambda_k(u)|^2}.
\end{equation} 
Consequently, in this case, projection DPPs are the only globally most repulsive DPPs. Such a process has a fixed number of points which agrees with the rank of the kernel. 

\subsection{Examples}\label{s:examplesDPP}

This section shows specific examples of our measure $p_u$ and the distribution of a point in $\xi_u$.
 
\subsubsection{DPPs defined on a finite set}\label{ex:1} 
Assume $\Lambda=\{1,\ldots,n\}$ is finite and $\nu$ is the counting measure; this is the simplest situation. Then $L^2(\Lambda)\equiv\mathbb C^n$, the class of possible kernels for DPPs corresponds to the class of $n\times n$ complex covariance matrices with all eigenvalues $\le1$, and the eigenfunctions simply correspond to normalized eigenvectors for such matrices. For simplicity we only consider projection DPPs and Poisson processes below, but other examples of DPPs on finite sets include uniform spanning trees (Example~14 in \cite{Hough:etal:09}) and  finite DPPs converging to the continuous Airy process on the complex plane \cite{Johansson}. 

The projection DPPs are given by complex projection matrices, ranging between the degenerated cases where $X=\emptyset$ and $X=\Lambda$. For example, consider the projection kernel of rank two given by 
$K(v,w)=\frac{1}{n}+t_v\overline{t_w}$, where $\sum_{i=1}^n t_i=0$ and $\sum_{i=1}^n|t_i|^2=1$. For any $u\in\{1,\ldots,n\}$, we have $p_u=1$ and 
\[\rho_{u}(v)=\frac{|\frac{1}{n}+t_u\overline{t_v}|^2}{\frac{1}{n}+|t_u|^2},\qquad v\in\{1,\ldots,n\},\]
is a probability mass function. This shows the repulsive effect of having a point of $X$ at $u$; in particular, $\rho_{u}(v)$
has a global maximum point at $v=u$.

The kernel of a Poisson process with intensity function $\rho\le1$ and conditioned on having no multiple points is given by a diagonal covariance matrix with diagonal entries $\rho(1),\ldots,\rho(n)$. If $\rho(u)>0$, then $p_u=\rho(u)$. This is a much different result as when we consider a Poisson process $X$ on a space $\Lambda$ where the reference measure $\nu$ is diffuse: If $\rho(u)>0$, then $p_u=0$ and almost surely $X=X^u$.
 
\subsubsection{Ginibre point processes}\label{ex:2}
From the standard Ginibre point process given by \eqref{e:Ginibre-standard-kernel}, other stationary point processes can be obtained. 
%Russ' comment about the 3 steps reduced to 2 steps caused the following change of lines in this paragrapgh.
Independently thinning the process with a retention probability $\alpha\beta$, where 
$\beta>0$ and $\alpha\in(0,1/\beta]$,
and multiplying each of the retained points by $\sqrt\beta$ gives a new stationary DPP 
%Third, multiplying the kernel of this new DPP with a parameter $\alpha\in(0,1/\beta]$, this 
%result in a stationary DPP $X$ 
with kernel
\begin{equation}\label{e:Kalphabeta}
%K_{\alpha,\beta}
K(v,w)=\frac\alpha\pi 
\exp\left(\frac{v\overline{w}}{\beta}-\frac{|v|^2+|w|^2}{2\beta}\right),\qquad v,w\in\mathbb C.
\end{equation}
%If $f_u=\rho_{\kappa_u}/p_u$ denotes the density of a point in $\kappa_u=X\setminus X^u$, we 
We have 
\begin{equation}\label{e:d1}
\rho={\alpha}/{\pi},\qquad p_u=\alpha\beta,\qquad f_u(v)=\frac{\exp\left(-|v-u|^2/\beta\right)}{\pi\beta}\sim N_\mathbb{C}(u,\beta). %,\qquad v\in\mathbb C.
\end{equation} 
The case where $\alpha = 1$ and $0 <\beta \leq 1$ is mentioned in Goldman's paper \cite{Goldman}, and the results in \eqref{e:d1} match those in Remark~24 in \cite{Goldman}. \cite{Dengetal} called the DPP with kernel \eqref{e:Kalphabeta} the scaled
$\beta$-Ginibre point process but the bound $\alpha\beta\le1$ was not noticed. For any fixed value of $\rho>0$, as the value of $\beta$ increases to its maximum $\min\{1,1/(\pi\rho)\}$, the more repulsive the process becomes, whilst as $\beta$ decreases to 0, in the limit a Poisson process with intensity $\rho$ is obtained. %\textcolor{red}{(The later result follows as the kernels converge in the weak operator topology to the Dirac kernel operator, and so Proposition~3.11 in \cite{shirai:takahash:03} implies that their Laplace functionals converge to the Laplace functional of a Poisson point process with intensity $\rho$.)}
  
\subsubsection{DPPs on $\mathbb R^d$ with a stationary kernel}\label{ex:3}
Suppose $\Lambda=\mathbb R^d$, $\nu$ is the Lebesgue measure, and $K(u,v)=K_0(u-v)$ is stationary, where $K_0\in L^2(\mathbb R^d)$ and $K_0$ is continuous. Then it follows from Parseval's identity that $p_u=1$ if and only if $\hat K_0$ is an indicator function whose integral agrees with the intensity of $X$, cf.\ Appendix J in \cite{LMR12extended}. A natural choice for the support of this indicator function is a ball centred at the origin in $\mathbb R^d$, and if (as in the standard Ginibre point process)
we let the intensity be $1/\pi$, then the globally most repulsive DPP has a stationary and isotropic kernel given by
\begin{equation}\label{e:jinc-kernel-general-dim-standard}
%K_{\mathrm{spec}}
K(v,w) =\int_{|y|^d\le d\Gamma(d/2)/(2\pi^{1+d/2})}\exp\left(2\pi i (v-w)\cdot y\right)\,\mathrm dy
, \qquad v, w \in \mathbb R^d,
\end{equation}
where $x\cdot y$ denotes the usual inner product for $x,y\in\mathbb R^d$ and $|y|$ is the usual Euclidean distance. For instance,
for $d=1$ this kernel is the sinc function and for $d=2$ it is the jinc-like function
\begin{equation}\label{e:jinc-kernel-2-dim-standard}
K%_{\mathrm{spec}}
(v,w) = J_1(2|v-w|)/(\pi|v-w|),
\end{equation}
where $J_1$ is the Bessel function of order one. We straightforwardly obtain the following proposition, where 
the moments in \eqref{e:kth-moments} follow from Eq.\ 10.22.57 in \cite{DLMF}. 

\begin{prop}\label{p2}
For the globally most repulsive DPP on $\mathbb R^d$ with kernel given by \eqref{e:jinc-kernel-general-dim-standard} and for any $u \in \mathbb{C}$, we have that $\rho_{u}(v)=\pi|K(u,v)|^2$ is a probability density function. In particular, for $d=2$, 
\[\rho_{u}(v) =  J_1(2|v - u|)^2 /\left(\pi |v-u|^2\right),\qquad  v\in\mathbb R^2,\]
and the moments of $|Z_u-u|$ with $Z_u\sim \rho_u$ %\textcolor{red}{\rho_{\kappa_u}}\textcolor{blue}{\rho_u}$ 
satisfy 
\begin{equation}\label{e:kth-moments}
\mathrm{E}\left(|Z_u - u|^k\right) =  
\frac{\Gamma(1+k/2)\Gamma(1-k)}{\Gamma(2-k/2)\Gamma(1 - k/2)^{2}},
%\beta\left(1 + \tfrac{k}{2}, 1 - k\right) /\,\Gamma\left(1 - \tfrac{k}{2}\right)^{2}, 
\qquad k \in (-2,1),
\end{equation}
and are infinite for $k \geq 1$.
\end{prop}

For comparison consider a standard Ginibre point process, where we can define $Z_u$ in a similar way as in Proposition~\ref{p2}. In both cases, $|Z_u-u|$ is independent of $(Z_u-u)/|Z_u-u|$, which is uniformly distributed on the unit circle. However, the distribution of $|Z_u-u|$ is very different in the two cases:
For the standard Ginibre point process, $|Z_u-u|^2$ is exponentially distributed and $|Z_u-u|$ has a finite $k$-th moment for all $k>-2$ given by $\Gamma(1+k/2)/(\pi\rho)^{k/2}$; whilst for the DPP on $\mathbb R^2$ with jinc-like kernel \eqref{e:jinc-kernel-2-dim-standard}, 
 $|Z_u-u|$ is heavy-tailed and has infinite $k$-th moments for all $k \geq 1$.

For any DPP $X$ with kernel $K$ and defined on $\mathbb R^d$, using independent thinning and scale transformation procedures similar to those in Section~\ref{ex:2} (replacing $\sqrt\beta$ by $\beta^{1/d}$ when transforming the points in the thinned process), we obtain a new DPP with kernel
\[K_{\mathrm{new}}(v,w)=\alpha K(v/\beta^{1/d},w/\beta^{1/d}),\qquad v,w\in\mathbb R^d,\]
where $\beta\in(0,1]$ and $\alpha\in(0,1/\beta]$. 
%The existence of this new DPP follows from the fact that the eigenvalues of $K_{\mathrm{new}}$ restricted to a compact set $S\subset\mathbb R^d$ are given by $\alpha\beta\lambda_k^{S/\beta^{1/d}}\in[0,1]$, where $\lambda_k^{S/\beta^{1/d}}$ is an eigenvalue for $K_{S/\beta^{1/d}}$. 
For instance, 
if $K$ is the jinc-like kernel for the globally most repulsive DPP given by \eqref{e:jinc-kernel-2-dim-standard}, the new DPP satisfies
the same equations for its intensity $\rho$ and its probability $p_u$ as in \eqref{e:d1}. Hence, if $\rho$ and $\beta$ are the same for this new DPP and the scaled $\beta$-Ginibre point process, the two DPPs are equally repulsive in terms of $p_u$. However, the probability density function for the point in $\xi_u$ conditioned on $\xi_u\not=\emptyset$ now becomes
\begin{equation}\label{e:d2}
f_u(v)=J_1\left(2|v-u|^2/\beta\right)/\left(\pi|v-u|^2/\beta\right).
\end{equation} 
The reach of the repulsive effect of the point at $u$ is much different when comparing the densities in \eqref{e:d1} and \eqref{e:d2}, in particular if $\beta$ is large. See Figure \ref{fig:density_comp} for a comparison of the densities \eqref{e:d1} and \eqref{e:d2} for $\beta = 1$.

    \begin{figure}
        \centering
        \includegraphics{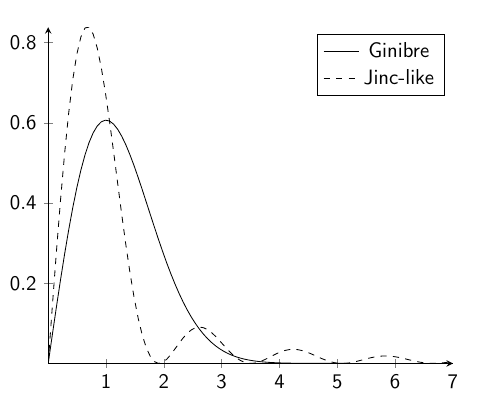}
        \caption{The densities of $|Z_0|$ for $Z_0 := X\backslash X^0$ for two globally most repulsive DPPs.}
        \label{fig:density_comp}
    \end{figure}

%and was first introduced in \cite{LMR15} using a spectral approach. 

\subsubsection{DPPs on $\mathbb S^d$ with an isotropic kernel}\label{ex:4}

Suppose $\Lambda=\mathbb S^d$ is the $d$-dimensional unit sphere, $\nu$ is the Lebesgue measure, and $K(v,w)=K_0(v\cdot w)$ is isotropic for all $v,w\in\mathbb S^d$. Then the DPP with kernel $K$ is isotropic, and $\rho=K_0(1)$ and $p_u$ do not depend on the choice of $u\in\Lambda$.
%, which we may let be the "North pole" $u=(0,...,1)$. 
By a classical result of Schoenberg \cite{MR0005922} and by Theorem~4.1 in \cite{Moelleretal18}, we have the following. The normalized eigenfunctions will be complex spherical harmonic functions, and $K_0$ will be real and of the form
\[K_0(t)=\rho\sum_{\ell=0}^\infty\beta_{\ell,d}\frac{C_\ell^{\left(\frac{d-1}{2}\right)}(t)}{C_\ell^{\left(\frac{d-1}{2}\right)}(1)},\qquad -1\le t\le 1,\]
where $C_\ell^{\left(\frac{d-1}{2}\right)}$ is a Gegenbauer polynomial of degree $\ell$ and the sequence $\beta_{0,d},\beta_{1,d},\ldots$ is a probability mass function. Further, letting $\sigma_d=\nu(\mathbb S^d)=2\pi^{(d+1)/2}/\Gamma((d+1)/2)$, 
the eigenvalues of $K$ are 
\[\lambda_{\ell,d}=\rho\sigma_d\beta_{\ell,d}/m_{\ell,d},\qquad \ell=0,1,\ldots,\]
with multiplicities
\[m_{0,1},\qquad m_{\ell,1}=2,\qquad \ell=1,2,\ldots,\qquad \mbox{if }d=1, \]
and
\[m_{\ell,d}=\frac{2\ell+d-1}{d-1}\frac{(\ell+d-2)!}{\ell!(d-2)!},\qquad \ell=0,1,\ldots,\qquad \mbox{if }d\in\{2,3,\ldots\}.\]
 So the DPP exists if and only if $\rho\le\inf_{\ell:\,\beta_{\ell,d}>0}m_{\ell,d}/(\sigma_d\beta_{\ell,d})$. Now, applying \eqref{e:compact p}, we obtain
\begin{equation}\label{e:gammasphere}
p_u=\rho\sigma_d\sum_{\ell=0}^\infty\beta_{\ell,d}^2/m_{\ell,d}.
\end{equation} 

There is a lack of flexible parametric DPP models on the sphere where $K_0$ is expressible in closed form, see Section~4.3 in \cite{Moelleretal18}. For instance, let $d=2$ and consider the special case of the multiquadric model given by
\[K_0(t)=\rho\frac{1-\delta}{\sqrt{1+\delta^2-2\delta t}},\qquad -1\le t\le 1,\]
with $\delta\in(0,1)$ a parameter and $0<\rho\le 1/(4\pi(1-\delta))$. Then, as shown in Section~4.3.2 in \cite{Moelleretal18}, the sequence
\begin{equation}\label{e:betal2}
\beta_{\ell,2}=(1-\delta)\delta^\ell,\qquad\ell=0,1,\ldots,
\end{equation}
specifies a geometric distribution and
\[\lambda_{\ell,2}= 4\pi\rho\delta^\ell(1-\delta)/(2\ell+1)\le \delta^\ell/(2\ell+1),\qquad\ell=0,1,\ldots.\]
As $\delta\rightarrow0$, then $\lambda_{0,2}\rightarrow4\pi\rho$ and $\lambda_{\ell,2}\rightarrow0$ if $\ell\ge1$, corresponding to the uninteresting case of a DPP with at most one point if $\rho<1/(4\pi)$ and with exactly one point if $\rho=1/(4\pi)$. From \eqref{e:gammasphere} and \eqref{e:betal2} we obtain
\[p_u=4\pi\rho(1-\delta)/(1+\delta)\le1/(1+\delta)
%\rho(4\pi)^2(1-\delta)^2\sum_{\ell=0}^\infty\frac{\delta^{2\ell}}{2\ell+1}=
%\rho\frac{(4\pi)^2(1-\delta)^2}{2\delta}\left(\ln(1+\delta)-\ln(1-\delta)\right)\le 
%\frac{2\pi(1-\delta)}{\delta}\left(\ln|1+\delta|-\ln|1-\delta|\right)
,\]
with this upper bound obtained for 
 the maximal value of $\rho= 1/(4\pi(1-\delta))$.
%\[\gamma=\frac{4\pi(1-\delta)}{2\delta}\left(\ln|1+\delta|-\ln|1-\delta|\right).\] 
Therefore the DPP with the multiquadric kernel is far from being globally most repulsive unless the expected number of points is very small.

Instead a flexible parametric model for the eigenvalues $\lambda_{\ell,d}$ is suggested in Section~4.3.4 in \cite{Moelleretal18} so that globally most repulsive DPPs as well as Poisson processes are obtained as limiting cases. However, the disadvantage of that model is that we can only numerically calculate $\rho$ and $p_u$.

\subsubsection{Remark}\label{s:rem} The considerations in Sections~\ref{s:measure of repulsiveness} and \ref{ex:1}-\ref{ex:4} 
are strictly for DPPs. For example, the intensity function of a Gibbs point process can be both smaller and larger than the intensity function of its Palm distribution at a given point; 
whilst for a DPP, $\rho\ge\rho^u$. 
Furthermore,
as a candidate for a 
`globally most repulsive stationary Gibbs point process on $\mathbb R^2$', we may
consider $Y=L_Z:=\{x+Z:x\in L\}$, where $L$ is the vertex set of a regular triangular lattice (the centres of a honeycomb structure) with one lattice point at the origin, and where $Z$ is a uniformly distributed point in the hexagonal region given by the Voronoi cell of the lattice and centred at the origin
(in other words, $Y$ may be considered as the limit of a stationary Gibbs hard core process when the packing fraction of hard discs increases to the maximal value $\approx0.907$, see e.g.\ \cite{Dogeetal04,Maseetal01}). However, the reduced Palm process at $u\in\mathbb R^2$ will be degenerated and given by $Y^u=L_u\setminus\{u\}$, which is a much different situation as compared to DPPs.

\Acks
%Jesper M\o ller was supported in part by The Danish Council for Independent Research | Natural Sciences, grant 7014-00074B, "Statistics for point processes in space and beyond", and 
%by the "Centre for Stochastic Geometry and Advanced Bioimaging",
%funded by grant 8721 from the Villum Foundation. Eliza O'reilly was supported in part by a grant of the Simons Foundation (\#197982 to UT Austin), the National Science Foundation Graduate Research Fellowship under Grant No. DGE-1110007, and the Department of Mathematical Sciences, Aalborg University.
% We are grateful to Professor Russell Lyon for helpful comments and discussions.
 Jesper M\o ller was supported in part by The Danish Council for Independent Research | Natural Sciences, grant 7014-00074B, `Statistics for point processes in space and beyond', and 
by the `Centre for Stochastic Geometry and Advanced Bioimaging',
funded by grant 8721 from the Villum Foundation. Eliza O'Reilly was supported in part by a grant of the Simons Foundation (\#197982 to UT Austin), the National Science Foundation Graduate Research Fellowship under Grant No. DGE-1110007, and the Department of Mathematical Sciences, Aalborg University.
% \textcolor{red}{We are grateful to Professor Russell Lyons for helpful comments and discussions.} 
%\textcolor{blue}{
We are grateful to Professor Russell Lyons for many useful comments, in particular for changing our focus on a more complicated coupling result established
in an earlier version of our paper (briefly, to obtain the reduced Palm process,
one point in the DPP was either moved or removed) to the simpler coupling
result in Theorem \ref{t:main-result}, which indeed is more suited for studying repulsiveness in
DPPs. Also, he pointed our attention to his paper \cite{Lyons:14}, which is essential in the
proof of Theorem \ref{t:main-result}. Finally, we are grateful to two referees, in particular for pointing our attention to the work by Pemantle and Peres.%}

\bibliographystyle{apt}
\bibliography{references}
\end{document}